\documentclass{article}
\usepackage[a4paper]{geometry}
\usepackage{graphicx}
\usepackage{subfigure}
\usepackage{url}
\usepackage{amsmath}
\usepackage{amssymb}
\usepackage{amsfonts}
\usepackage{amsthm}
\usepackage{enumerate,paralist}
\usepackage[textsize=footnotesize]{todonotes}
\usepackage{xspace}
\usepackage{hyperref}

\newtheorem{theorem}{Theorem}
\newtheorem{lemma}{Lemma}
\newtheorem{corollary}{Corollary}

\newtheorem{property}{Property}

\title{Edge Partitions of Optimal\\ $2$-plane and $3$-plane Graphs
\thanks{%
An extended abstract of the research presented in this paper has been accepted at the $44$th International Workshop on Graph-Theoretic Concepts in Computer Science (WG 2018).}
\thanks{%
Research funded in part by the project: ``Algoritmi e sistemi di analisi visuale di reti complesse e di grandi
dimensioni" - Ricerca di Base 2017, Dip. Ing. Univ. Perugia.}
}
\author{
Michael Bekos$^1$,
Emilio Di Giacomo$^2$,
Walter Didimo$^2$,\\
Giuseppe Liotta$^2$,
Fabrizio Montecchiani$^2$, 
Chrysanthi Raftopoulou$^3$
\\[0.1in]
$^1$Institut f{\"u}r Informatik, Universit{\"a}t T{\"u}bingen, T{\"u}bingen, Germany 
\\\texttt{\small bekos@informatik.uni-tuebingen.de} 
\and
$^2$Dip. Ingegneria, Universit\`a degli Studi di Perugia, Perugia, Italy 
\\\texttt{\small\{emilio.digiacomo,walter.didimo,giuseppe.liotta,fabrizio.montecchiani\}@unipg.it} 
\and
$^3$Dep. Mathematics, National Technical University of Athens, Athens, Greece
\\\texttt{\small crisraft@mail.ntua.gr}
}

\begin{document}

\maketitle

\begin{abstract}
A topological graph is a graph drawn in the plane. 
A topological graph is $k$-plane, $k>0$, if each edge is crossed at most $k$ times. 
We study the problem of partitioning the edges of a $k$-plane graph such that each partite set forms a graph with a simpler structure. While this problem has been studied for $k=1$, we focus on \emph{optimal} $2$-plane and $3$-plane graphs, which are $2$-plane and $3$-plane graphs with maximum~density. We prove the following results. (i)~It is not possible to partition the edges of a simple optimal $2$-plane graph into a $1$-plane graph and a forest, while (ii)~an edge partition formed by a $1$-plane graph and two plane forests always exists and can be computed in linear time. (iii)~We describe efficient algorithms to partition the edges of a simple optimal $2$-plane graph into a $1$-plane graph and a plane graph with maximum vertex degree $12$, or with maximum vertex degree $8$ if the optimal $2$-plane graph is such that its crossing-free edges form a graph with no separating triangles. (iv)~We exhibit an infinite family of simple optimal $2$-plane graphs such that in any edge partition composed of a $1$-plane  graph  and a plane graph,  the plane graph has maximum vertex degree at least $6$ and the $1$-plane graph has maximum vertex degree at least $12$. (v)~We show that every optimal $3$-plane graph whose crossing-free edges form a biconnected graph can be decomposed, in linear time, into a $2$-plane graph and two plane forests.
\end{abstract}

\section{Introduction}\label{se:intro}

Partitioning the edges of a graph such that each partite set induces a subgraph with a simpler structure is a fundamental problem in graph theory with various applications, including the design of graph drawing algorithms. For example, a classic result by Schnyder~\cite{DBLP:conf/soda/Schnyder90} states that the edge set of any maximal planar graph can be partitioned into three trees, which can be used to efficiently compute planar straight-line drawings on a grid of quadratic size. Edge partitions of planar graphs have also been studied by Gon{\c{c}}alves~\cite{DBLP:conf/stoc/Goncalves05}, who proved that the edges of every planar graph can be partitioned into two outerplanar graphs, thus solving a conjecture by Chartrand et al.~\cite{Chartrand197112}, and improving previous results by  Elmallah~\cite{ec-pepg+-88}, Kedlaya~\cite{Kedlaya1996238}, and Ding et al.~\cite{Ding2000221}.
More in general, there exist various graph parameters based on edge partitions. 
For example, the \emph{arboricity} of a graph $G$ is the minimum number of forests needed to cover all edges of $G$, while $G$ has \emph{thickness} $t$ if it is the union of $t$ planar graphs. Durocher and Mondal~\cite{DBLP:conf/icalp/DurocherM16} studied the interplay between the thickness $t$ of a graph and the number of bends per edge in a drawing that can be partitioned into~$t$~planar~sub-drawings.

Recently, edge partitions have been studied for the family of \emph{$1$-planar graphs}. A graph is \emph{$k$-planar} ($k \geq 1$)  if it can be drawn in the plane such that each edge is crossed at most $k$ times~\cite{DBLP:journals/combinatorica/PachT97}; a topological graph is \emph{$k$-plane} if it has at most $k$ crossings per edge. The $k$-planar graphs represent a natural extension of planar graphs, and fall within the more general framework of~\emph{beyond planarity}. Beyond planarity studies graph families that admit drawings where some prescribed edge crossing configurations are forbidden~\cite{JGAA-459,DBLP:journals/corr/abs-1804-07257,hong_et_al,shonan}.  
Ackerman~\cite{DBLP:journals/dam/Ackerman14} proved that the edges of a $1$-plane graph can be partitioned into a plane graph (a topological graph with no crossings) and a plane forest, extending an earlier result by Czap and Hud\'ack~\cite{DBLP:journals/combinatorics/CzapH13}. A $1$-planar graph with $n$ vertices is \emph{optimal} if it contains exactly $4n-8$ edges, which attains the maximum density for $1$-planar graphs. Lenhart et al.~\cite{DBLP:journals/tcs/LenhartLM17} proved that every optimal $1$-plane graph can be partitioned into two plane graphs such that one has maximum vertex degree four, where the bound on the vertex degree is worst-case optimal. Di Giacomo et al.~\cite{algo18} proved that every triconnected (not necessarily optimal) $1$-plane graph can be partitioned into two plane graphs such that one has maximum vertex degree six, which is also a tight bound. This result was exploited to show that every such graph admits a visibility representation in which the vertices are orthogonal polygons with few reflex corners each, while the edges are horizontal and vertical lines of sight between vertices~\cite{algo18}. Additional results on edge partitions of various subclasses of $1$-plane graphs are~reported~in~\cite{DBLP:journals/tcs/GiacomoDELMMW18}.

\begin{figure}[t]
\centering
\subfigure[]{\includegraphics[width=0.47\columnwidth,page=7]{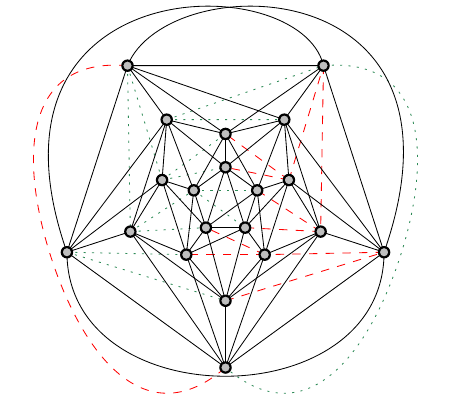}\label{fi:intro-2}}\hfil
\subfigure[]{\includegraphics[width=0.47\columnwidth,page=4]{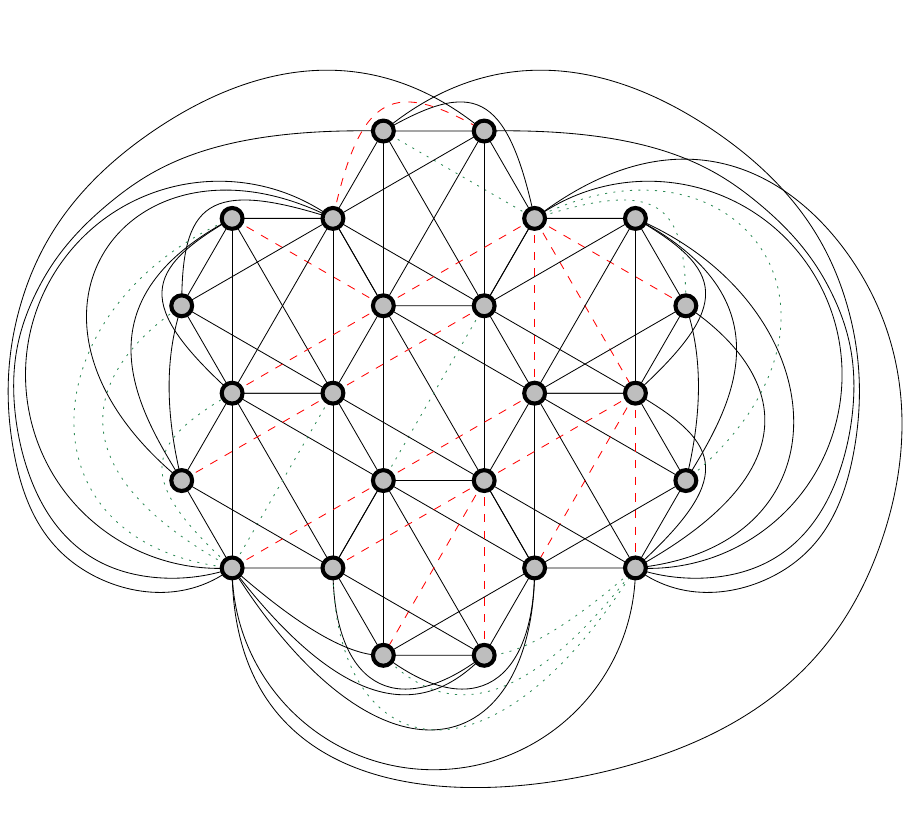}\label{fi:intro-4}}
\caption{An edge partition of: (a)  an optimal $2$-plane graph into a $1$-plane graph (solid) and two plane forests (dashed and dotted);  (b) an optimal $3$-plane graph into a $2$-plane graph (solid) and two plane forests (dashed and dotted).
}
\end{figure}

While $1$-planar graphs have been extensively studied (for a recent survey refer to~\cite{DBLP:journals/csr/KobourovLM17}), and their structure has been deeply understood, this is not the case for $2$-planar and $3$-planar graphs. These graphs have at most $5n-10$ edges and $5.5n-11$ edges~\cite{DBLP:journals/combinatorica/PachT97}, respectively, and their structure is more complex. Similarly to $1$-planar graphs, a $2$-planar (respectively, $3$-planar) graph with $n$ vertices is \emph{optimal} if it contains exactly $5n-10$ (respectively, $5.5n-11$) edges; see also Section~\ref{se:preliminaries} for formal definitions. Examples of optimal $2$-plane and optimal $3$-plane graphs are shown in Figs.~\ref{fi:intro-2} and~\ref{fi:intro-4}, respectively. Bekos et al.~\cite{DBLP:conf/compgeom/Bekos0R17} recently characterized optimal $2$-planar and optimal $3$-planar graphs, and showed that these graphs have a regular structure; refer to Section~\ref{se:preliminaries} for details. In this paper, we build upon this characterization and we initiate the study of edge partitions of simple (i.e., with neither self-loops nor parallel edges) optimal $2$-plane graphs. We then extend some of our results to a subclass of optimal $3$-plane graphs. More precisely, our contributions are as follows.

\begin{itemize}
\item We prove that it is not possible to partition the edges of a simple optimal $2$-plane graph $G$ into a $1$-plane graph and a forest (see Theorem~\ref{thm:no1planeforest}). 
Note that, by Nash-Williams formula~\cite{JLMS:JLMS0445}, $2$-planar graphs have arboricity at most five, while $1$-planar graphs have arboricity at most four. Hence, our result implies that in a decomposition of $G$ into five forests, it is not possible to pick four of them forming a $1$-plane graph. 

\item On the positive side, every optimal $2$-plane graph, whose crossing-free edges form a biconnected graph, can be partitioned into a $1$-plane graph and two plane forests. This implies that every simple optimal $2$-plane graph admits such an edge partition (see Theorem~\ref{thm:orientation}). This result exploits some insights in the structure of optimal $2$-plane graphs; also, the edge partition can be computed in linear time. For an example, refer to Fig.~\ref{fi:intro-2}. 

\item Additionally, we prove that the edges of a simple optimal $2$-plane graph can always be partitioned into a $1$-plane graph and a plane graph with maximum vertex degree $12$ (see Theorem~\ref{thm:upper_bound}). The upper bound on the vertex degree can be lowered to $8$ if the crossing-free edges of the optimal $2$-plane graph form a graph with no separating triangles (see Theorem~\ref{thm:upper-bound-2}). Both bounds are achieved with constructive techniques in polynomial time.
 
\item Besides the above upper bound on the vertex degree, we establish a non-trivial lower bound. Specifically, we exhibit an infinite family of simple optimal $2$-plane graphs such that in any edge partition composed of a $1$-plane graph and a plane graph, the plane graph has maximum vertex degree at least $6$ and the $1$-plane graph has maximum vertex degree at least $12$ (see Theorem~\ref{thm:lower_bound}). 

\item We finally consider (non-simple) optimal $3$-plane graphs and  prove that any such a topological graph, whose crossing-free edges form a biconnected graph, can be partitioned into a $2$-plane graph and two plane forests; also in this case, the edge partition can be computed in linear time (see Theorem~\ref{thm:orientation-3planar}). For an example, refer to Fig.~\ref{fi:intro-4}.
\end{itemize}
 
Section~\ref{se:preliminaries} contains preliminaries and notation. In Section~\ref{se:2planar}, we present our results for optimal $2$-plane graphs, while in Section~\ref{se:3planar} we extend them to optimal $3$-plane graphs. A concluding discussion followed by a set of interesting open problems can be found in Section~\ref{se:open}. 

\section{Preliminaries and Notation}\label{se:preliminaries}
\noindent\textbf{Drawings and planarity.} A graph is \emph{simple} if it contains neither  self-loops nor  parallel edges. A \emph{drawing} of a graph $G=(V,E)$ is a mapping of the vertices of $V$ to points of the plane, and of the edges of $E$ to Jordan arcs connecting their corresponding endpoints but not passing through any other vertex. We only consider \emph{simple} drawings, i.e., drawings such that two arcs representing two edges have at most one point in common, and this point is either a common endpoint or a common interior point where the two arcs properly cross each other. A graph drawn in the plane is also called a \emph{topological graph}. The \emph{crossing graph} $C(G)$ of a topological graph $G$ has a vertex for each edge of $G$ and an edge between two vertices if and only if the two corresponding edges of $G$ cross each other. A topological graph is \emph{plane} if it has no edge crossings.  A plane graph subdivides the plane into topologically connected regions, called \emph{faces}. The unbounded region is the \emph{outerface}. The \emph{length} of a face $f$ is the number of vertices encountered in a closed walk along the boundary of $f$. If a vertex $v$ is encountered $k>0$ times, then $v$ has \emph{multiplicity} $k$ in $f$. In a biconnected graph, all vertices have multiplicity one in the faces they belong to.

\begin{figure}[t]
\centering
\subfigure[]{\includegraphics[width=0.45\columnwidth,page=5]{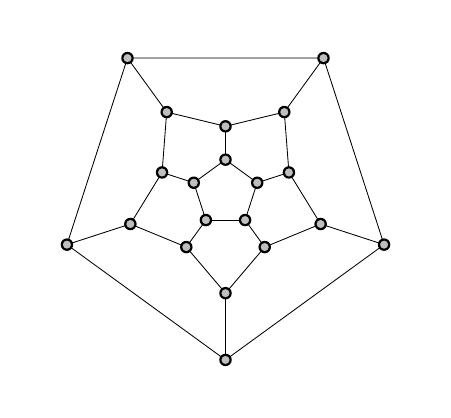}\label{fi:prel-1}}
\subfigure[]{\includegraphics[width=0.45\columnwidth,page=8]{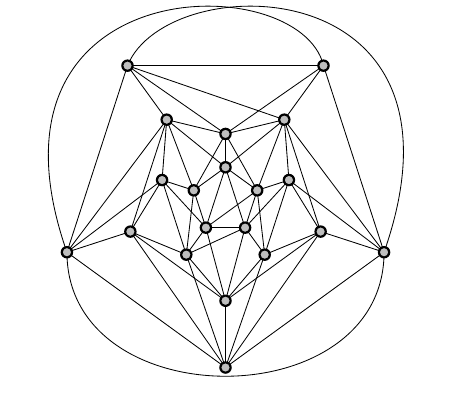}\label{fi:prel-2}}
\caption{(a) The pentangulation of the graph $G$ in Fig.~\ref{fi:intro-2}. (b) A $1$-plane graph obtained from $G$ by removing two adjacent chords from each filled pentagon.}
\end{figure}

\smallskip\noindent\textbf{$k$-planar graphs.} A topological graph is \emph{$k$-plane} if each edge is crossed at most $k$ times.  A \emph{pentangulation} $P$ (respectively, \emph{hexangulation} $H$) is a plane graph such that all its faces are $5$-cycles (respectively, $6$-cycles), which we call \emph{pentagons} (respectively, \emph{hexagons}). Two parallel edges are \emph{homotopic} if either the interior or the exterior region bounded by their curves contains no vertices. A self loop is \emph{homotopic} if the interior or the exterior region bounded by its curve contains no vertices. Bekos et al.~\cite{DBLP:conf/compgeom/Bekos0R17} proved that an $n$-vertex graph $G$ is \emph{optimal $2$-planar} if and only if it admits a drawing without homotopic self-loops and parallel edges, such that the graph formed by the crossing-free edges is a pentangulation $P(G)$ with $n$ vertices, and each face of $P(G)$ has five crossing edges in its interior, which we call \emph{chords} in the following. Also, each chord has exactly two crossings. A pentagon with its five chords routed as described above will be called a \emph{filled pentagon}. Fig.~\ref{fi:prel-1} shows the pentangulation $P(G)$ of the optimal $2$-plane graph $G$ of Fig.~\ref{fi:intro-2}. Similarly, Bekos et al.\ proved that an $n$-vertex graph $G$ is optimal $3$-planar if and only if it admits a drawing  without homotopic self-loops and homotopic parallel edges, such that the graph formed by the crossing-free edges is a hexangulation $H(G)$ with $n$ vertices, and each face of $H(G)$ has eight crossing edges in its interior, which we also call chords. Accordingly, a hexagon with its eight chords routed as described above will be called a \emph{filled hexagon}.

\smallskip\noindent\textbf{Arboricity and orientations.} The \emph{arboricity} of a graph is the minimum number of forests into which its edges can be partitioned. Nash-Williams~\cite{JLMS:JLMS0445} proved that a graph $G$ has arboricity $a \ge 1$ if and only if, $a = \max \{\lceil \frac{m_S}{n_S-1} \rceil\}$ over all subgraphs $S$ of $G$ with $n_S \ge 2$ vertices and $m_S$ edges. A \emph{$d$-orientation} of a graph $G$ is an orientation of the edges of $G$ such that each vertex has at most $d$ outgoing edges, for some integer $d \ge 1$. Note that if a graph has arboricity $a$, then it admits an $a$-orientation (while the converse may not be true). 
Given two vertices $s$ and $t$ of a graph $G$, an \emph{$st$-orientation} of $G$ is an orientation of its edges such that $G$ becomes a directed acyclic graph with a single source $s$ and a single~sink~$t$. 

\smallskip\noindent\textbf{Edge partitions.} Given a topological graph $G=(V,E)$, an \emph{edge partition} of $G$ is denoted by $\langle E_1, \dots, E_p \rangle$, for some $p>1$, where $E = E_1 \cup \dots \cup E_p$ and $E_i \cap E_j = \emptyset$ ($1 \le i \neq j \le p$). We denote by $G[E_i]$ the topological graph obtained from $G$ by removing all edges not in $E_i$ and all isolated vertices.
	
\section{Edge Partitions of Optimal $2$-plane Graphs}\label{se:2planar}

We begin with the following property, which will be useful in the remainder of this section.

\begin{property}\label{pr:chords}
	Let $G'=(V,E \setminus R)$ be a topological graph obtained by removing a subset $R$ of crossing edges from a simple optimal $2$-plane graph $G=(V,E)$. Graph $G'$ is $1$-plane if and only if $R$ has (at least) two adjacent chords for each filled pentagon of $G$.  	
\end{property}
\begin{proof}
The crossing graph $C(G)$ of $G$ contains a cycle of length $5$ for each filled pentagon of $G$, where adjacent chords of $G$ correspond to vertices of $C(G)$ that belong to the same cycle and are at distance two. Notice that $G'$ is $1$-plane if and only if its crossing graph $C(G')$ has maximum vertex degree one  and thus if and only if the edges in $R$ correspond to a dominating set of $C(G)$. Since a dominating set of $C(G)$ has at least two non adjacent vertices from each cycle, $R$ contains at least two adjacent chords for each filled pentagon of $G$.
\end{proof}

\noindent For example, the graph in Fig.~\ref{fi:prel-2} is obtained by removing two adjacent chords from each filled pentagon of an optimal $2$-plane graph.

\subsection{Edge partitions with acyclic subgraphs}\label{sse:acyclic}

As already mentioned, the edge set of a $1$-plane graph can always be partitioned into a plane graph and a plane forest~\cite{DBLP:journals/dam/Ackerman14}. One may wonder whether this result can be generalized to $2$-plane graphs, that is, whether the edge set of every $2$-plane graph can be partitioned into a $1$-plane graph and a forest. Theorem~\ref{thm:no1planeforest} shows that this may not be always possible. In particular, this is never the case for optimal $2$-plane graphs. On the positive side, Theorem~\ref{thm:orientation} gives a constructive technique to partition the edges of every optimal $2$-plane graph into a $1$-plane graph and two plane forests (rather than one).

\begin{theorem}\label{thm:no1planeforest}
Let $G$ be a simple optimal $2$-plane graph. 
Graph $G$ has no edge partition $\langle E_1, E_2 \rangle$ such that $G[E_1]$ is a $1$-plane graph and $G[E_2]$ is a forest.
\end{theorem}
\begin{proof}
Consider an edge partition $\langle E_1, E_2 \rangle$ such that $G[E_1]$ is a $1$-plane graph. 
By Property~\ref{pr:chords}, $E_2$ contains at least two chords from each filled pentagon of $G$. 
By Euler's formula, if $G$ has $n$ vertices, then the pentangulation $P(G)$ has $\frac{2}{3}(n-2)$ faces, and thus $G$ has  $\frac{2}{3}(n-2)$ filled pentagons. 
It follows that $E_2$ contains at least $2 \times \frac{2}{3}(n-2)$ edges, and hence $G[E_2]$ can be a forest only if $n\le 5$. 
However, since $G$ is simple, $n>5$ holds, as otherwise $P(G)$ would be a $5$-cycle and $G$ would have five pairs of parallel chords.
\end{proof}

In order to prove Theorem~\ref{thm:orientation}, we first prove an analogous result for a wider family of optimal $2$-plane graphs, and then show that this family contains all simple optimal $2$-plane graphs.

\begin{lemma}\label{le:orientation}
Every $n$-vertex optimal $2$-plane graph $G=(V,E)$ whose pentangulation $P(G)$ is biconnected has an edge partition $\langle E_1, E_2, E_3 \rangle$, which can be computed in $O(n)$ time, such that $G[E_1]$ is a $1$-plane graph and both $G[E_2]$ and $G[E_3]$ are plane forests.
\end{lemma}
\begin{proof}
To construct the desired edge partition, we first appropriately select the edges of $E_1$ so that $E'=E \setminus E_1$ contains two adjacent chords for each filled pentagon of $G$. By Property~\ref{pr:chords}, this implies that $G[E_1]$ is a $1$-plane graph. We then color the edges of $E'$  with two colors, say green  and red, so that each monochromatic set is a plane forest. The set of green  edges will correspond to $E_2$, while the set of red edges to $E_3$. 

We aim at computing an $st$-orientation of $P(G)$. Recall that, given a biconnected plane graph and two vertices $s$ and $t$ on its outerface, it is possible to construct an $st$-orientation of the graph in linear time (see, e.g.,~\cite{DBLP:journals/tcs/EvenT76,DBLP:journals/dcg/TamassiaT86}). Since $P(G)$ is biconnected, we can compute an $st$-orientation of $P(G)$ (with $s$ and $t$ on the outerface). According to this orientation, all outgoing edges of any vertex $v\in P(G)$ appear consecutively around $v$, followed by all the incoming edges of $v$ (\cite[Lemma~2]{DBLP:journals/dcg/TamassiaT86}). For any vertex $v\in P(G)$ distinct from $s$ and $t$, this allows us to uniquely define the \emph{leftmost}  (\emph{rightmost}) face of $v$ as the face containing the last incoming and first outgoing edges (last outgoing and first incoming edges, respectively) of $v$ in clockwise order around $v$. We use this fact to classify the internal faces of $P(G)$ in different types. 
By~\cite[Lemma~1]{DBLP:journals/dcg/TamassiaT86}, each internal face $f$ of $P(G)$ has a source vertex $s(f)$ and a target vertex $t(f)$, and consists of two directed paths from $s(f)$ to $t(f)$, say $p_l(f)$ and $p_r(f)$. Since $P(G)$ is a pentangulation, we have $|p_l(f)|+|p_r(f)|=5$, $|p_l(f)| \leq 4$, and  $|p_r(f)|\leq 4$. We say that $f$ is a face of type $i-j$ if $|p_l(f)|=i$ and $|p_r(f)|=j$. Hence, in total there exist exactly four different \emph{types} of internal faces: $1-4$, $4-1$, $2-3$ and $3-2$; refer to Fig.~\ref{fig:orientation}. For each internal face of $P(G)$, we select two adjacent chords and we add them to $E'$ as follows. For the first two types of faces, we select the two chords of $f$ in $G$ that are incident to the target vertex $t(f)$. In the other two types, we select the two edges that are incident to the middle vertex of the directed path with edge-length $2$. If $i<j$ (respectively, $i>j$) we color the selected edges red (respectively, green). Note that we have not selected and colored any chord of the outerface; this selection will be made at the very end.

\begin{figure}[t]
		\centering
		\subfigure[$1-4$]{\includegraphics[width=0.15\columnwidth,page=1]{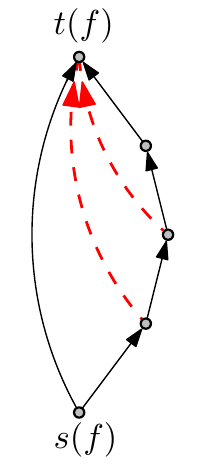}\label{fig:orientation_1}}\hfil
		\subfigure[$4-1$]{\includegraphics[width=0.15\columnwidth,page=2]{figures/orientation}\label{fig:orientation_2}}\hfil
		\subfigure[$2-3$]{\includegraphics[width=0.15\columnwidth,page=3]{figures/orientation}\label{fig:orientation_3}}\hfil
		\subfigure[$3-2$]{\includegraphics[width=0.15\columnwidth,page=4]{figures/orientation}\label{fig:orientation_4}}
		\caption{Illustration for Theorem~\ref{thm:orientation}. Red (respectively, green) edges are dashed (respectively, dotted).\label{fig:orientation}}
	\end{figure}
	
We now claim that each monochromatic subgraph induced by the red and green edges is a (simple) forest. We prove this claim for the red subgraph, symmetric arguments hold for the green one. We orient each pair of red edges of every interior face $f$ of $P(G)$ towards their common end-vertex. Observe that if $(u,v)$ is a directed red edge in a face $f$ from $u$ to $v$, then $f$ is the leftmost face of $u$. Since, the leftmost face of each vertex is unique, it follows that every vertex has at most one outgoing red edge. Hence, a cycle of red edges would be actually a directed cycle (otherwise it would contain at least one vertex with out-degree two, contradicting the previous statement). Consider the plane subgraph $G_{red}$ of $G$ containing the edges of $P(G)$ (oriented according to the $st$-orientation defined above) and the red edges (each pair oriented towards the common end-vertex). We show that $G_{red}$ does not contain directed cycles, which implies that the red subgraph is a forest. We actually prove a stronger property of $G_{red}$, namely, we show that the orientation of $G_{red}$ is an $st$-orientation. The proof is by induction on the number $i\ge 0$ of internal faces of $P(G)$ having red chords in $G_{red}$. If $i=0$, the statement trivially follows since $G_{red}$ corresponds to $P(G)$. Assume that the claim holds for $i \ge 0$, and suppose there are $i+1$ internal faces of $P(G)$ having red chords in $G_{red}$. Consider any such face $f$ of $P(G)$, and let $G'_{red}$ be the graph obtained from $G_{red}$  by removing the two red chords of $f$. Graph $G'_{red}$ is $st$-oriented by the inductive hypothesis. Obviously, reinserting the two removed chords in $G'_{red}$ creates neither new sources nor new sinks. Moreover, reinserting the two chords cannot create a directed cycle, since each reinserted chord $(u,v)$ connects either vertices on opposite paths of face $f$, or $v=t(f)$. In the first case, there cannot be a directed path in $G'_{red}$ from $v$ to $u$ by~\cite[Lemma~4]{DBLP:journals/dcg/TamassiaT86}. In the second case, observe that there is already a directed path from $u$ to $v$ in $G'_{red}$ and thus there cannot be a directed path from $v$ to $u$ because $G'_{red}$ is acyclic.  
	
It remains to select and color two chords of $G$ from the outerface of $P(G)$. For each interior face $f$ of $P(G)$, red or green edges are never incident to the source vertex $s(f)$. Hence, there is neither a red nor a green edge incident to $s$ (which is the source of the $st$-orientation). We arbitrarily select one of the two chords of $G$ in the outerface of $P(G)$ that is incident to $s$ to be red and the other one to be green. Since the degree of $s$ in the red (respectively, green) subgraph is equal to one, it follows that no cycle is created. Furthermore, the red (respectively, green) subgraph is simple, because every vertex has at most one outgoing red edge (respectively, green edge) and thus two parallel edges would form a cycle. Since an $st$-orientation can be computed in $O(n)$ time, and since $G$ has $O(n)$ faces and $O(n)$ edges, the theorem follows. Fig.~\ref{fi:intro-2} shows an edge partition computed with the described algorithm.
\end{proof}

\noindent In the following lemma, we prove that every simple optimal $2$-plane graph belongs to the family of graphs that satisfy the conditions of Lemma~\ref{le:biconnected}.

\begin{lemma}\label{le:biconnected}
The pentangulation $P(G)$ of a simple optimal $2$-plane graph $G$ is biconnected.
\end{lemma}
\begin{proof}
Suppose to the contrary that there exists a cutvertex $v$ of $P(G)$. Then there exists a face $f$ of $P(G)$ such that a closed walk along the boundary of $f$ encounters $v$ at least twice, i.e., $v$ has multiplicity greater than one in $f$. Also, according to the characterization of optimal $2$-plane graphs in~\cite{DBLP:conf/compgeom/Bekos0R17}, $f$ has length $5$ and it induces a complete graph $K_5$ in $G$. Since $v$ has multiplicity greater than one in $f$, it follows that there exists a self-loop at $v$; a contradiction, since $G$ is simple.
\end{proof}

\noindent The next theorem follows directly from Lemma~\ref{le:orientation} and Lemma~\ref{le:biconnected}.

\begin{theorem}\label{thm:orientation}
Every $n$-vertex simple optimal $2$-plane graph $G=(V,E)$ has an edge partition $\langle E_1, E_2, E_3 \rangle$, which can be computed in $O(n)$ time, such that $G[E_1]$ is a $1$-plane graph and both $G[E_2]$ and $G[E_3]$ are plane forests.
\end{theorem}

\noindent Finally, combining Theorem~\ref{thm:orientation} with a result by Ackerman~\cite{DBLP:journals/dam/Ackerman14}, stating that the edges of a $1$-plane graph can be partitioned into a plane graph and a plane forest, we obtain the following.

\begin{corollary}
Every simple optimal $2$-plane graph has an edge partition $\langle E_1$, $E_2$, $E_3$, $E_4 \rangle$ such that $G[E_1]$ is a plane graph, and  $G[E_i]$ is a plane forest,~for~$i\ge 2$.
\end{corollary}

\subsection{Edge partitions with bounded vertex degree subgraphs}\label{sse:degree}

We now prove that the edge set of a simple optimal $2$-plane graph can be partitioned into a $1$-plane graph and a plane graph whose maximum vertex degree is bounded by a small constant. An analogous result holds for optimal $1$-plane graphs~\cite{DBLP:journals/tcs/LenhartLM17}. 
We will make use of the following technical lemma.

\begin{lemma}\label{le:pentagon}
Let $v_0,v_1,v_2,v_3,v_4$ be the (distinct) vertices of a $5$-cycle $C$ in clockwise order starting from $v_0$. Let the edges of $C$ be arbitrarily oriented.  
There exists an index $0 \le j \le 4$ such that each of the three vertices $v_j$, $v_{j+2}$, $v_{j+3}$ (indexes taken modulo $5$) is incident to at least one outgoing edge of the $5$-cycle.
\end{lemma}
\begin{proof}
We distinguish a few cases based on the number $k$ of vertices of $C$ incident to two  incoming edges. Refer to Fig.~\ref{fig:5cycle} for an illustration. 
\begin{figure}[t]
\centering
\subfigure[]{\includegraphics[width=0.3\columnwidth,page=1]{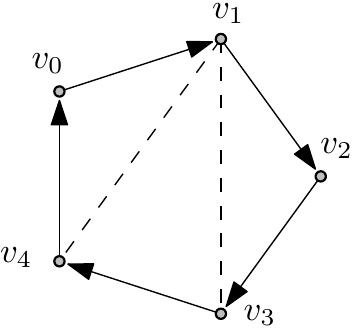}}\hfil
\subfigure[]{\includegraphics[width=0.3\columnwidth,page=2]{figures/5cycle}}\hfil
\subfigure[]{\includegraphics[width=0.3\columnwidth,page=3]{figures/5cycle}}
\caption{Illustration for Lemma~\ref{le:pentagon}. (a) $k=0$ and $j=1$; (b) $k=1$ and $j=3$; (c) $k=2$ and $j=4$.\label{fig:5cycle}}
\end{figure}
\begin{itemize}
\item If $k=0$, the claim trivially follows since any index $0 \le j \le 4$ satisfies the statement. 
\item If $k=1$, let $0 \le h \le 4$ be the index of the only vertex $v_h$ incident to two incoming edges of $C$. Then both $j=h-1$ and $j=h+1$ satisfy the statement. 
\item If $k=2$, the two vertices $v_h$ and $v_{h'}$ incident to two incoming edges of $C$ cannot be adjacent, and hence the vertex $v_j$ that is adjacent to both of them satisfies the statement.  
\end{itemize} 
Note that the case $k>2$ is not possible, as there would be an edge that is incoming with respect to both its end-vertices. This concludes our proof.
\end{proof}

\begin{theorem}\label{thm:upper_bound}
Every $n$-vertex simple optimal $2$-plane graph $G=(V,E)$ has an edge partition $\langle E_1, E_2 \rangle$, which can be computed in $O(n)$ time, such that $G[E_1]$ is a $1$-plane graph and $G[E_2]$ is a plane graph of maximum vertex degree $12$.
\end{theorem}
\begin{proof}
We construct the desired edge partition as follows. Remove three chords from every pentagon of $P(G)$ such that the resulting graph $G'$ is plane and all its faces have length three. Compute a $3$-orientation of $G'$ in linear time, by using the algorithm in~\cite{DBLP:conf/soda/Schnyder90}. From now on, we assume that the edges of $P(G)$ are directed according to this $3$-orientation. 
For each filled pentagon of $G$ we select three vertices that satisfy the conditions of Lemma~\ref{le:pentagon}, and we mark to be part of $E_2$ the two chords of the pentagon incident to the selected vertices. All other edges are part of $E_1$. Since each vertex has at most three outgoing edges in the $3$-orientation of $P(G)$, and each of these edges is shared by exactly two pentagons (as otherwise $G$ would be non-simple), we have that each vertex is selected for at most six pentagons and therefore is incident to at most $12$ edges~in~$E_2$. 

By Property~\ref{pr:chords}, graph $G[E_1]$ is $1$-plane. Also, graph $G[E_2]$ has maximum vertex degree 12 as shown above, and no two edges of $G[E_2]$ cross, because either they share an end-vertex or they are inside different pentagons of $P(G)$.
\end{proof}

The upper bound of Theorem~\ref{thm:upper_bound} can be improved if $P(G)$ has no separating triangles. The next lemma can be proved by using the Nash-Williams formula.

\hyphenation{arboricity}

\begin{lemma}\label{lem:arb-two}
A simple pentangulation with no separating triangles has~arboricity two.
\end{lemma}
\begin{proof}
Let $P$ be a simple pentangulation with no separating triangles. Consider any subgraph $S$ of $P$ with $n_S$ vertices and $m_S$ edges. 
We claim that any face of $S$ has length at least four. Suppose, for a contradiction, that $S$ contains a face $f$ with length smaller than four. Since $P$ does not contain parallel edges, the length of $f$ cannot be two. Hence, we assume that $f$ has length three. Then either $f$ is a face of $P$ as well, or $f$ is a separating triangle of $P$. Both cases contradict our assumption.

Consequently, all faces of $S$ have length at least four, by Euler's formula and by the Handshaking Lemma we conclude that $S$ has $m_S \le \frac{4(n_S-2)}{2} \le 2n_S$ edges. On the other hand $P$ has more than $n-1$ edges. Thus, by the Nash-Williams formula, the arboricity of $P$ is two.
\end{proof}

\begin{theorem}\label{thm:upper-bound-2}
Every $n$-vertex simple optimal $2$-plane graph $G=(V,E)$ whose pentangulation $P(G)$ has no separating triangles has an edge partition $\langle E_1, E_2 \rangle$, which can be computed in $O(n^{1.5})$ time, such that $G[E_1]$ is a $1$-plane graph and $G[E_2]$ is a plane graph of maximum vertex degree $8$.
\end{theorem}
\begin{proof}
By Lemma~\ref{lem:arb-two}, $P(G)$ has arboricity at most two, and hence admits a $2$-orientation. Asahiro et al.~\cite{DBLP:journals/ijfcs/AsahiroMOZ07} proved that, for a given $k$, a $k$-orientation (if it exists) of an (unweighted) graph with $m$ edges can be computed in $O(m^{1.5})$ time. It follows that we can use the algorithm in~\cite{DBLP:journals/ijfcs/AsahiroMOZ07} to compute a $2$-orientation of $P(G)$ in $O(n^{1.5})$ time.
We then proceed similarly as for the proof of Theorem~\ref{thm:upper_bound} in order to construct the desired edge partition. For each pentagon we select three vertices that satisfy the conditions of Lemma~\ref{le:pentagon}, and we mark the two chords of the pentagon incident to the selected vertices to be part of $E_2$. Since each vertex has at most two outgoing edges in the $2$-orientation of $P(G)$, and each of these edges is shared by exactly two pentagons, we have that each vertex is selected for at most four pentagons and therefore is incident to at most $8$~marked~edges.
\end{proof}

The next corollary is a consequence of Theorems~\ref{thm:upper_bound} and~\ref{thm:upper-bound-2}, together with the fact that every $3$-connected $1$-plane graph can be decomposed into a plane graph and a plane graph with maximum vertex degree six~\cite{algo18}.

\begin{corollary}
Every $n$-vertex simple optimal $2$-plane graph $G$ has an edge partition $\langle E_1, E_2, E_3 \rangle$, which can be computed in $O(n)$ time, such that $G[E_1]$ is plane, $G[E_2]$ is plane with maximum vertex degree $12$, and $G[E_3]$ is plane with maximum vertex degree $6$.  Also, if $P(G)$ has no separating triangles, then $G$ has an edge partition $\langle E_1, E_2, E_3 \rangle$, which can be computed in $O(n^{1.5})$ time, such that $G[E_1]$ is plane, $G[E_2]$ is plane with maximum vertex degree $8$, and $G[E_3]$ is plane with maximum vertex degree $6$. 
\end{corollary}
\begin{proof}
By Theorem~\ref{thm:upper_bound}, one can compute in $O(n)$ time an edge partition $\langle E'_1, E_2 \rangle$ such that $G[E'_1]$ is $1$-plane and $G[E_2]$ is a plane graph with maximum vertex degree $12$. Also, by Theorem~\ref{thm:upper-bound-2}, if $P(G)$ has no separating triangles, one can compute in $O(n^{1.5})$ time an edge partition $\langle E'_1, E_2 \rangle$ of $E$ such that $G[E'_1]$ is $1$-plane and $G[E_2]$ is a plane graph with maximum vertex degree $8$.

On the other hand, Di Giacomo et al.~\cite{algo18} proved that the edges of a triconnected $1$-plane graph can be partitioned into a plane graph and a plane graph with maximum vertex degree $6$ in linear time. Thus, it suffices to show that $G[E'_1]$ is triconnected. Recall that, both the algorithm of Theorem~\ref{thm:upper_bound} and the algorithm of Theorem~\ref{thm:upper-bound-2} partition the edges so that $E_2$ contains two adjacent chords from each filled pentagon of $G$. It follows that the crossing-free edges of $G[E'_1]$ form a plane graph whose faces have either length three or four, and each face of length four contains two crossing chords in its interior (see, e.g., Fig.~\ref{fi:prel-2}). Then, $G[E'_1]$ has a triangulated plane graph as a spanning subgraph, and thus it is triconnected.
\end{proof}

We conclude this section by proving a lower bound for the maximum vertex degree of an edge partition into a $1$-plane graph and a plane graph or into two plane graphs.

\begin{theorem}\label{thm:lower_bound}
There exists an infinite family $\mathcal{G}$ of simple optimal $2$-plane graphs, such that in any edge partition $\langle E_1, E_2 \rangle$ of $G\in \mathcal{G}$ where $G[E_1]$ is $1$-plane and $G[E_2]$ is plane, $G[E_1]$ has maximum vertex degree at least $12$ and $G[E_2]$ has maximum vertex degree at least $6$.
\end{theorem}
\begin{proof}
\begin{figure}[th]
\centering
\subfigure[]{\includegraphics[width=0.3\columnwidth,page=4]{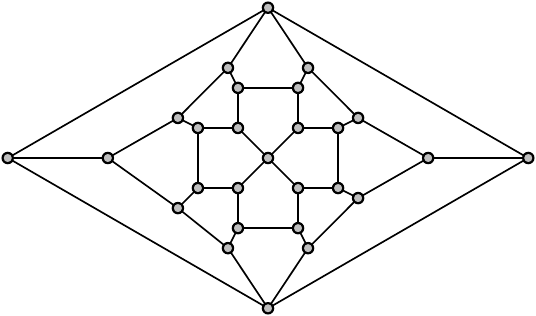}\label{fig:lower_bound_2}}\hfil
\subfigure[]{\includegraphics[width=0.3\columnwidth,page=5]{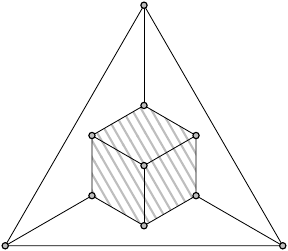}\label{fig:lower_bound_1}}\hfil
\subfigure[]{\includegraphics[width=0.3\columnwidth,page=6]{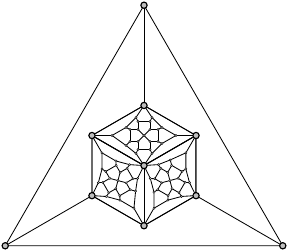}\label{fig:lower_bound_3}}
\caption{Illustrations for Theorem~\ref{thm:lower_bound}.}
\end{figure}
For every $n > 24$, we construct a graph $G_n$ with $O(n)$ vertices as described in the following. 
Consider the plane graph $G_1$ in Fig.~\ref{fig:lower_bound_2}. Note that all faces of $G_1$ have length five, except for the outer face which is a $4$-cycle. Construct the graph $G_2$ by gluing a copy of the graph $G_1$ in each of the three gray-tiled quandrangular faces of the graph in Fig.~\ref{fig:lower_bound_1}. Note that  all faces of $G_2$ have length five, except for the outer face which is a $3$-cycle. Then, starting from an $n$-vertex maximal plane graph $M_n$, identify each face of $M_n$ (including its outer face) with the outer face of a copy of $G_2$. Note that this operation is feasible, since all faces of $M_n$ are $3$-cycles by maximality. This results in a pentangulation $P_n$ with $O(n)$ vertices. $G_n$ is obtained by adding all five chords inside each pentagon of $P_n$. Graph $G_n$ is optimal $2$-plane as it satisfies the characterization in~\cite{DBLP:conf/compgeom/Bekos0R17}, and it is simple because $P_n$ is simple and triconnected. 

Consider now any edge partition $\langle E_1, E_2 \rangle$ of $G_n$, such that $G[E_1]$ is $1$-plane. Then, by Property~\ref{pr:chords}, $E_2$ contains at least two chords of each filled pentagon of $G_n$. Therefore, for each face of $M_n$, there are at least three edges of $E_2$ having one end-vertex in $M_n$ (at least one for each filled pentagon incident to the outer face of the copy of $G_2$ identified with this face). This means that $E_2$ contains at least $3(2n-4)=6n-12$ edges incident to vertices of $M_n$. Let $k$ be the maximum number of edges of $E_2$ that are incident to a single vertex of $M_n$. Then, we have $kn\geq 6n-12\Rightarrow k\geq 6$ for $n > 12$.

Similarly, since $G[E_2]$ is plane, $E_1$ contains at least three chords of each filled pentagon of $G_n$, and by the same argument used above,  we can conclude that there are at least six edges of $E_1$ having one end-vertex in $M_n$, and hence  $E_1$ contains at least $6(2n-4)=12n-24$ edges incident to vertices of $M_n$. Let $k$ be the maximum number of edges of $E_1$ that are incident to a single vertex of $M_n$, we have $kn\geq 12n-24\Rightarrow k\geq 12$ for $n > 24$.
\end{proof}

\section{Edge Partitions of Optimal $3$-plane Graphs}\label{se:3planar}

In this section, we study optimal $3$-plane graphs and we aim at showing the existence of a decomposition into a $2$-plane graph and two plane forests. It is known that no optimal $3$-plane graph is simple~\cite{DBLP:conf/compgeom/Bekos0R17}, and hence its hexangulation may also be non-simple. We show that a similar strategy as the one used in the proof of Theorem~\ref{thm:orientation} can be employed provided that the underlying hexangulation of the graph is biconnected and hence each of its faces is a simple $6$-cycle. 
Consider a filled hexagon $h$ of an optimal $3$-plane graph $G$. If the hexangulation $H(G)$ is biconnected, $h$ contains six distinct vertices, which we denote by $v_0,v_1,\dots,v_5$ following their clockwise order in a closed walk along the boundary of $h$; refer to Fig.~\ref{fig:hexagons-1}. We know  that $h$ contains $8$ chords (see Section~\ref{se:preliminaries}), and, in particular, there are only two vertices of $h$ that are not connected by an edge of $h$; we call these two vertices the \emph{poles} of $h$ (black in Fig.~\ref{fig:hexagons-1}). Let $v_i$ and $v_j$  ($0 \le i < j \le 5$) be the poles of $h$. Note that $j-i=3$, and that each chord of $h$ is crossed at most twice after removing one of the following \emph{patterns}: 

\begin{enumerate}
\item[$(\alpha)$] the two chords of $h$ incident to $v_i$ or to $v_j$; see Fig.~\ref{fig:hexagons-2},  
\item[$(\beta)$] one of the two \emph{$Z$-paths} $(v_i,v_{i+2})$, $(v_{i+2},v_{j+2})$, $(v_{j+2},v_j)$ and $(v_i,v_{j+1})$, $(v_{j+1},v_{i+1})$, $(v_{i+1},v_j)$, where indexes are taken modulo $6$; see Figs.~\ref{fig:hexagons-3}-\ref{fig:hexagons-4}, and 
\item[$(\gamma)$] any three adjacent chords of $h$; see Fig.~\ref{fig:hexagons-5}. 
\end{enumerate}

\begin{figure}[ht]
	\centering
	\subfigure[]{\includegraphics[width=0.18\columnwidth,page=1]{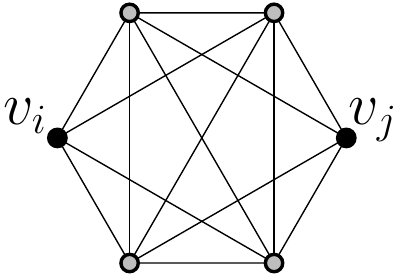}\label{fig:hexagons-1}}\hfil
	\subfigure[Pattern ($\alpha$)]{\includegraphics[width=0.18\columnwidth,page=2]{figures/hexagons}\label{fig:hexagons-2}}\hfil
	\subfigure[Pattern ($\beta$)]{\includegraphics[width=0.18\columnwidth,page=3]{figures/hexagons}\label{fig:hexagons-3}}\hfil
	\subfigure[Pattern ($\beta$)]{\includegraphics[width=0.18\columnwidth,page=4]{figures/hexagons}\label{fig:hexagons-4}}\hfil
	\subfigure[Pattern ($\gamma$)]{\includegraphics[width=0.18\columnwidth,page=5]{figures/hexagons}\label{fig:hexagons-5}}
	\caption{\label{fig:hexagons}(a) A filled hexagon (poles shown in black). (b)--(d) The three patterns. }
\end{figure}

\begin{theorem}\label{thm:orientation-3planar}
Every $n$-vertex optimal $3$-plane graph $G=(V,E)$ whose hexangulation $H(G)$ is biconnected has an edge partition $\langle E_1, E_2, E_3 \rangle$, which can be computed in $O(n)$ time, such that $G[E_1]$ is a $2$-plane graph, and both $G[E_2]$ and $G[E_3]$ are plane forests.
\end{theorem}
\begin{proof}
To construct the desired edge partition, we first guarantee that $E'=E \setminus E_1$ contains, for each filled hexagon of $G$, one of the three patterns described above, which implies that $G[E_1]$ is $2$-plane. We then color the edges of $E'$  with two colors, say green and red, so that each monochromatic set is a plane forest. The set of green edges will correspond to $E_2$, while the set of red edges to $E_3$. 

We compute an $st$-orientation of $H(G)$ by choosing a pole of the outerface as vertex $s$. Recall that each internal face $f$ of $H(G)$ has a source vertex $s(f)$ and a target vertex $t(f)$, and consists of two directed paths from $s(f)$ to $t(f)$, say $p_l(f)$ and $p_r(f)$. The number of edges $|p_l(f)|, |p_r(f)|$ of the two paths is at most $5$, and in particular $|p_l(f)|+|p_r(f)|=6$. We say that $f$ is a face of type $i-j$ if $|p_l(f)|=i$ and $|p_r(f)|=j$. Hence, in total there exist exactly five different \emph{types} of internal faces: $1-5$, $5-1$, $2-4$, $4-2$, $3-3$; refer to Fig.~\ref{fig:orientation-2}. For the first two types of faces we add to $E'$ the two or three chords of $f$ in $G$ that are incident to the target vertex $t(f)$ (i.e., we remove either pattern ($\alpha$) or ($\gamma$)). In the type $1-5$ (respectively, $5-1$) we color these edges red (respectively, green). For the types $2-4$ (respectively, $4-2$), we add to $E'$ the two or three chords incident to the middle vertex of $p_l(f)$ (respectively, $p_r(f)$), and we color them red (respectively, green). For the type $3-3$, we distinguish a set of cases based on the position of the poles. Suppose first that the poles are $s(f)$ and $t(f)$, then we add to $E'$ the two chords incident to $t(f)$ (pattern ($\alpha$)), and we color red (respectively, green) the chord incident to a vertex of $p_r(f)$ (respectively, $p_l(f)$); see Fig.~\ref{fig:orientation-hexa-5}. Otherwise, among the two possible $Z$-paths, there is one that does not contain neither $s(f)$ nor $t(f)$ (pattern ($\beta$)), and we remove it; see Fig.~\ref{fig:orientation-hexa-6}.

\begin{figure}[t]
	\centering
	\subfigure[]{\includegraphics[width=0.145\columnwidth,page=1]{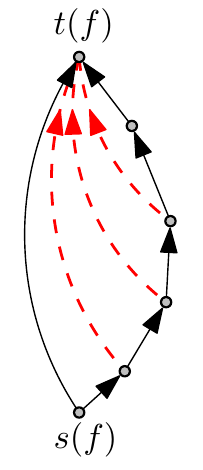}\label{fig:orientation-hexa-1}}\hfil
	\subfigure[]{\includegraphics[width=0.145\columnwidth,page=2]{figures/orientation-hexa}\label{fig:orientation-hexa-2}}\hfil
	\subfigure[]{\includegraphics[width=0.145\columnwidth,page=3]{figures/orientation-hexa}\label{fig:orientation-hexa-3}}\hfil
	\subfigure[]{\includegraphics[width=0.145\columnwidth,page=4]{figures/orientation-hexa}\label{fig:orientation-hexa-4}}\hfil
	\subfigure[]{\includegraphics[width=0.145\columnwidth,page=5]{figures/orientation-hexa}\label{fig:orientation-hexa-5}}\hfil
	\subfigure[]{\includegraphics[width=0.145\columnwidth,page=6]{figures/orientation-hexa}\label{fig:orientation-hexa-6}}
	\caption{Illustration for Theorem~\ref{thm:orientation-3planar}. Red (respectively, green) edges are dashed (respectively, dotted).\label{fig:orientation-2}}
\end{figure}

We now claim that each monochromatic subgraph induced by the red and green edges is a forest. We prove this claim for the red subgraph, symmetric arguments hold for the green one. We orient the edges such that all red (respectively, green) edges are outgoing with respect to their end-vertex belonging to $p_r(f)$ (respectively, $p_l(f)$); note that there is always such an end-vertex. This orientation implies that each vertex has at most one outgoing red edge, hence a cycle of red edges would be actually a directed cycle. Consider the plane subgraph $G_{red}$ of $G$ containing the oriented edges of $H(G)$  and the oriented red edges. Since each red edge in a face $f$ either connects a vertex of $p_r(f)$ to a vertex of $p_l(f)$, or it is incident to $t(f)$, a similar argument as the one used in the proof of Theorem~\ref{thm:orientation} shows that the orientation of $G_{red}$ is an $st$-orientation, and thus that there are no directed cycles.   
	
It remains to select and color two chords of $G$ from the outerface of $H(G)$. As in the proof of Theorem~\ref{thm:orientation}, there is neither a red nor a green edge incident to the vertex $s$ of the outerface, which is a pole by construction. We color red one of the two chords of the outerface incident to $s$, and we color green the other one (i.e., we remove pattern ($\alpha$) from the outerface). Since the degree of $s$ in the red (respectively, green) subgraph is equal to one, no cycle is created. An example is shown in Fig.~\ref{fi:intro-4}. Furthermore, the red (respectively, green) subgraph is simple, because every vertex has at most one outgoing red edge (respectively, green edge) and thus two parallel edges would form a cycle. Since an $st$-orientation can be computed in $O(n)$ time, and since $G$ has $O(n)$ faces and $O(n)$ edges, the theorem follows.
\end{proof}

\section{Discussion and open problems}\label{se:open}

We studied edge partitions of $2$-plane and $3$-plane graphs. 
We proved that simple optimal $2$-plane graphs cannot be partitioned into a $1$-plane graph and a forest, while an edge partition with a $1$-plane graph and two forests always exists. A natural question is whether the edges of a (simple) optimal $2$-plane graph can be partitioned into a plane graph and two forests. Moreover, the problem of partitioning the edges of an optimal $3$-plane graph into a $2$-plane graph and a forest is still open.

We showed that the edges of a simple optimal $2$-plane graph $G$ can be partitioned into a $1$-plane graph and a plane graph with maximum vertex degree $12$, or $8$ if the pentangulation of $G$ does not contain separating triangles.  On the other hand, there exist simple optimal $2$-plane graphs for which the plane graph of any such an edge partition has maximum vertex degree at least $6$. Reducing the gap between these two bounds on the vertex degree is an interesting problem. Also, can we improve the time complexity of Theorem~\ref{thm:upper-bound-2}?

\smallskip We conclude with a result that sheds some light on the structure of $k$-plane graphs, for any $k \ge 2$. While one can easily observe that every $k$-plane graph can be partitioned into a set of $k+1$ plane graphs, the next theorem shows a stronger property.

\begin{theorem}
Every $n$-vertex $k$-plane graph with $k \ge 2$ has an edge partition $\langle E_1,E_2 \rangle$, which can be computed in $O(k^{1.5} n)$ time, such that $G[E_1]$ is a plane graph and $G[E_2]$ is a $(k-1)$-plane graph.
\end{theorem}
\begin{proof}
Let $G=(V,E)$ be a $k$-plane graph, with $k \ge 2$. 
We compute a maximal independent set $I$ of the crossing graph $C(G)$ of $G$. Recall that a maximal independent set is also a dominating set~\cite{berge}. 
Consider now a partition $\langle E_1, E_2 \rangle$ of $E$ such that $E_2$ contains the edges of $G$ corresponding to the vertices in $I$, while $E_1$ contains all other edges of $G$. Since the vertices of $I$ form an independent set, the edges in $E_2$ do not cross each other, and thus $G[E_2]$ is a plane graph. Also, the vertices of $I$ are a dominating set of $C(G)$ and thus $G[E_1]$ is a graph such that each edge is crossed at most $k-1$ times, i.e., a $(k-1)$-plane graph, as desired.

Since a maximal independent set can be computed in linear time in the number of edges of $C(G)$, and since a $k$-plane graph has $O(k^{1.5} n)$ crossings~\cite{DBLP:journals/combinatorica/PachT97}, the statement follows. 
\end{proof} 

\medskip\noindent\textbf{Acknowledgments.} Research started at the 2017 GNV Workshop, held in Heiligkreuztal (Germany). We wish to thank the organizers of the workshop and all the participants for the fruitful atmosphere and the useful discussions.

\bibliography{paper}

\begin{thebibliography}{10}

\bibitem{DBLP:journals/dam/Ackerman14}
E.~Ackerman.
\newblock A note on 1-planar graphs.
\newblock {\em Discrete Appl. Math.}, 175:104--108, 2014.

\bibitem{DBLP:journals/ijfcs/AsahiroMOZ07}
Y.~Asahiro, E.~Miyano, H.~Ono, and K.~Zenmyo.
\newblock Graph orientation algorithms to minimize the maximum outdegree.
\newblock {\em Int. J. Found. Comput. Sci.}, 18(2):197--215, 2007.

\bibitem{JGAA-459}
M.~A. {Bekos}, M.~{Kaufmann}, and F.~{Montecchiani}.
\newblock Guest editors' foreword and overview.
\newblock {\em J. Graph Algorithms Appl.}, 22(1):1--10, 2018.

\bibitem{DBLP:conf/compgeom/Bekos0R17}
M.~A. Bekos, M.~Kaufmann, and C.~N. Raftopoulou.
\newblock On optimal 2- and 3-planar graphs.
\newblock In {\em {SoCG} 2017}, volume~77 of {\em LIPIcs}, pages 16:1--16:16.
  Schloss Dagstuhl - Leibniz-Zentrum fuer Informatik, 2017.

\bibitem{berge}
C.~Berge.
\newblock {\em Theory of graphs and its applications}.
\newblock Methuen \& Co., 1962.

\bibitem{Chartrand197112}
G.~Chartrand, D.~Geller, and S.~Hedetniemi.
\newblock Graphs with forbidden subgraphs.
\newblock {\em J. Combin. Theory Ser. B}, 10(1):12 -- 41, 1971.

\bibitem{DBLP:journals/combinatorics/CzapH13}
J.~Czap and D.~Hud{\'{a}}k.
\newblock On drawings and decompositions of 1-planar graphs.
\newblock {\em Electr. J. Comb.}, 20(2):P54, 2013.

\bibitem{DBLP:journals/tcs/GiacomoDELMMW18}
E.~{Di Giacomo}, W.~Didimo, W.~S. Evans, G.~Liotta, H.~Meijer, F.~Montecchiani,
  and S.~K. Wismath.
\newblock New results on edge partitions of 1-plane graphs.
\newblock {\em Theor. Comput. Sci.}, 713:78--84, 2018.

\bibitem{algo18}
E.~{Di Giacomo}, W.~Didimo, W.~S. Evans, G.~Liotta, H.~Meijer, F.~Montecchiani,
  and S.~K. Wismath.
\newblock Ortho-polygon visibility representations of embedded graphs.
\newblock {\em Algorithmica}, 80(8):2345--2383, 2018.

\bibitem{DBLP:journals/corr/abs-1804-07257}
W.~Didimo, G.~Liotta, and F.~Montecchiani.
\newblock A survey on graph drawing beyond planarity.
\newblock {\em CoRR}, abs/1804.07257, 2018.

\bibitem{Ding2000221}
G.~Ding, B.~Oporowski, D.~P. Sanders, and D.~Vertigan.
\newblock Surfaces, tree-width, clique-minors, and partitions.
\newblock {\em J. Combin. Theory Ser. B}, 79(2):221 -- 246, 2000.

\bibitem{DBLP:conf/icalp/DurocherM16}
S.~Durocher and D.~Mondal.
\newblock Relating graph thickness to planar layers and bend complexity.
\newblock In {\em {ICALP} 2016}, volume~55 of {\em LIPIcs}, pages 10:1--10:13.
  Schloss Dagstuhl - Leibniz-Zentrum fuer Informatik, 2016.

\bibitem{ec-pepg+-88}
E.~S. Elmallah and C.~J. Colbourn.
\newblock Partitioning the edges of a planar graph into two partial k-trees.
\newblock In {\em Congressus Numerantium}, pages 69--80, 1988.

\bibitem{DBLP:journals/tcs/EvenT76}
S.~Even and R.~E. Tarjan.
\newblock Computing an \emph{st}-numbering.
\newblock {\em Theor. Comput. Sci.}, 2(3):339--344, 1976.

\bibitem{DBLP:conf/stoc/Goncalves05}
D.~Gon{\c{c}}alves.
\newblock Edge partition of planar graphs into two outerplanar graphs.
\newblock In {\em {STOC} 2005}, pages 504--512. {ACM}, 2005.

\bibitem{hong_et_al}
S.~Hong, M.~Kaufmann, S.~G. Kobourov, and J.~Pach.
\newblock {Beyond-Planar Graphs: Algorithmics and Combinatorics (Dagstuhl
  Seminar 16452)}.
\newblock {\em Dagstuhl Reports}, 6(11):35--62, 2017.

\bibitem{shonan}
S.~Hong and T.~Tokuyama.
\newblock {Algorithmics for Beyond Planar Graphs}.
\newblock \url{http://shonan.nii.ac.jp/shonan/blog/2015/11/15/3972/}.
\newblock {NII} Shonan Meeting, Shonan Village Center, 2016.

\bibitem{Kedlaya1996238}
K.~S. Kedlaya.
\newblock Outerplanar partitions of planar graphs.
\newblock {\em J. Combin. Theory Ser. B}, 67(2):238 -- 248, 1996.

\bibitem{DBLP:journals/csr/KobourovLM17}
S.~G. Kobourov, G.~Liotta, and F.~Montecchiani.
\newblock An annotated bibliography on 1-planarity.
\newblock {\em Computer Science Review}, 25:49--67, 2017.

\bibitem{DBLP:journals/tcs/LenhartLM17}
W.~J. Lenhart, G.~Liotta, and F.~Montecchiani.
\newblock On partitioning the edges of 1-plane graphs.
\newblock {\em Theor. Comput. Sci.}, 662:59--65, 2017.

\bibitem{JLMS:JLMS0445}
C.~S.~A. Nash-Williams.
\newblock Edge-disjoint spanning trees of finite graphs.
\newblock {\em J. London Math. Soc.}, s1-36(1):445--450, 1961.

\bibitem{DBLP:journals/combinatorica/PachT97}
J.~Pach and G.~T{\'{o}}th.
\newblock Graphs drawn with few crossings per edge.
\newblock {\em Combinatorica}, 17(3):427--439, 1997.

\bibitem{DBLP:conf/soda/Schnyder90}
W.~Schnyder.
\newblock Embedding planar graphs on the grid.
\newblock In {\em {SoDA} 1990}, pages 138--148. {SIAM}, 1990.

\bibitem{DBLP:journals/dcg/TamassiaT86}
R.~Tamassia and I.~G. Tollis.
\newblock A unified approach to visibility representations of planar graphs.
\newblock {\em Discrete Comput. Geom.}, 1:321--341, 1986.

\end{thebibliography}
\bibliographystyle{abbrv}

\end{document}